\documentclass[12pt]{amsart}

\newtheorem{theorem}{Theorem}[section]
\newtheorem{lemma}[theorem]{Lemma}

\newtheorem{proposition}[theorem]{Proposition}
\newtheorem{remark}[theorem]{Remark}
\newtheorem{definition}[theorem]{Definition}

\newcommand{\ncom}{\newcommand}
\ncom{\rar}{\rightarrow}
\ncom{\lrar}{\longrightarrow}
\ncom{\ov}{\overline}
\ncom{\m}{\mbox}
\ncom{\sta}{\stackrel}
\renewcommand{\sp}{{\rm sp}}
\ncom{\comx}{{\mathbb C}}
\ncom{\Z}{{\mathbb Z}}
\ncom{\Q}{{\mathbb Q}}
\ncom{\R}{{\mathbb R}}
\ncom{\G}{{\mathbb G}}
\ncom{\al}{\alpha}
\ncom{\p}{{\mathbb P}}
\ncom{\E}{{\mathbb E}}
\ncom{\N}{{\mathbb N}}
\ncom{\K}{{\mathbb K}}
\ncom{\Le}{{\mathbb L}}
\ncom{\A}{{\mathbb A}}
\ncom{\B}{{\mathbb B}}
\ncom{\F}{{\mathbb F}}
\ncom{\C}{{\mathbb C}}
\ncom{\f}{\frac}
\ncom{\cA}{{\mathcal A}}
\ncom{\cX}{{\mathcal X}}
\ncom{\cO}{{\mathcal O}}
\ncom{\cW}{{\mathcal W}}
\ncom{\cL}{{\mathcal L}}
\ncom{\cP}{{\mathcal P}}
\ncom{\cH}{{\mathcal H}}
\ncom{\cS}{{\mathcal S}}
\ncom{\cM}{{\mathcal M}}
\ncom{\cC}{{\mathcal C}}
\ncom{\cT}{{\mathcal T}}
\ncom{\cF}{{\mathcal F}}
\ncom{\cN}{{\mathcal N}}
\ncom{\cJ}{{\mathcal J}}
\ncom{\cV}{{\mathcal V}}
\ncom{\cZ}{{\mathcal Z}}
\ncom{\cU}{{\mathcal U}}
\ncom{\cSU}{{\mathcal S \mathcal U}}
\ncom{\cG}{{\mathcal G}}
\ncom{\cQ}{{\mathcal Q}}
\ncom{\cR}{{\mathcal R}}
\ncom{\cE}{{\mathcal E}}
\ncom{\cY}{{\mathcal Y}}
\ncom{\cD}{{\mathcal D}}
\ncom{\cK}{{\mathcal K}}

\begin{document}
\baselineskip=16pt

\title{Degeneration of the modified diagonal cycle}

\author[J. N. Iyer]{J. N. Iyer}
\author[S. M\"uller--Stach]{S. M\"uller--Stach}

\address{The Institute of Mathematical Sciences, CIT
Campus, Taramani, Chennai 600113, India}
\email{jniyer@imsc.res.in}

\address{Mathematisches Institut der Johannes Gutenberg Universit\"at
Mainz, Staudingerweg 9, 55099 Mainz, Germany}
\email{mueller-stach@uni-mainz.de}


\footnotetext{Mathematics Classification Number: 14C25, 14D05, 14D20,
 14D21}

\begin{abstract}
In this note, we revisit the modified diagonal cycle $\Delta_e \in CH_1(C\times C\times C)$ 
of Gross and Schoen. We look at degenerations of this cycle, induced by a degeneration of the curve $C$, and explain 
how the specialization map with respect to the central fiber produces a higher Chow cycle. 
When $C$ is non-hyperelliptic of genus three, and degenerates to a nodal curve, 
the degeneration of the cycle is an indecomposable higher Chow cycle 
$\Delta_{e,1}\in CH^2(C'\times C',1)\otimes \Q$, where $C'$ is hyperelliptic of genus two. 
\end{abstract}
\maketitle


\setcounter{tocdepth}{1}
\tableofcontents
 
\section{Introduction}
The modified diagonal cycle
$$
\Delta_e := \Delta_{123} -\Delta_{12}-\Delta_{13}-\Delta_{23} + \Delta_{1}+ \Delta_2 + \Delta_3 \in CH_1(C\times C\times C)
$$ 
is a one-dimensional cycle on a triple product of a smooth connected projective curve $C$ over a field $k$ which is an alternating sum of diagonal type cycles. 
We assume $k=\C$. It was shown in \cite{GrossSchoen}, that $\Delta_e$ gives a non-trivial element in the Griffiths group of one cycles modulo algebraic equivalence, 
if the curve $C$ is non-hyperelliptic.

We would like to understand the variation of the modified diagonal cycle, when a genus three non-hyperelliptic  curve degenerates to a single nodal curve. 
The philosophy of S. Bloch \cite{Bloch} says that the degeneration is a higher Chow cycle.
This type of degeneration was first looked at by A. Collino \cite{Collino} in the case of a Jacobian of a smooth non-hyperelliptic curve. 
He studied the variation of Ceresa's cycle, when the curve degenerates to a nodal curve. If the normalization is a hyperelliptic curve $C$, 
he associated a higher Chow cycle in $CH^2(J(C),1)$, and proved it to be indecomposable.

Here, we show that the outcome is very close to Collino's computations. 
Suppose $\cC\rar B$ is family of projective curves of genus three over $B:={\rm Spec}(R)$. Here $R$ is a discrete valuation ring, and the generic fibre is smooth, 
the special fibre is a single nodal curve $C'_0$. Assume there is a section $\tilde{e}$ for this family and $C'\rar C'_0$ is the normalization.
Then by \cite[p.664]{GrossSchoen}, there is a good family of triple products $\pi:\cY\rar B$ together with the modified diagonal cycle $ \Delta_{e}$ on $\cY$.
We show the following.

\begin{theorem} \label{maintheorem}
The specialization of the cycle $\Delta_e$ on the special fibre of $\pi$ corresponds to a higher Chow cycle 
$\Delta_{e,1}$ in $CH^2(C'\times C',1)\otimes \Q$. It is indecomposable, if $C'$ is generic.
\end{theorem}

See \S \ref{degcycle}, Theorem \ref{degener}. Note that a higher Chow cycle on $CH^2(C'\times C',1)\otimes \Q$ is 
indecomposable, if it is not in the image of the natural map
$$
{\rm Pic}(C'\times C')_\Q \otimes \C^* \longrightarrow CH^2(C'\times C',1)_\Q.
$$
The method and the proof can be extended to a further degeneration of the family $\cC\rar B$, where the normalization of the special fibre is an elliptic curve $E$, 
to obtain a higher Chow cycle $\Delta_{e,2}$ in $CH^2(E,2)\otimes \Q$.
This is a multiple of the Collino's degenerated cycle, up to algebraic equivalence. 
See Proposition \ref{elliptic}. We even further degenerate $\Delta_{e,2}$ to a cycle $\Delta_{e,3}$ in $CH^2({\rm Spec}(\C),3)\otimes \Q$.
This is related to Collino's work in \cite[Remark 7.13]{Collino}, and made precise in \cite[IV D]{GGKerr}, by Green-Griffiths-Kerr, via limit Abel-Jacobi maps. 

The technical tools to make the degeneration precise in the cycle constructions, are provided by specialization maps (see Proposition \ref{specialization}), and Levine's relative 
$K_0$ groups (see \cite{Levine2}, \S \ref{LevineK}). We recall them and their relation to higher Chow groups in \S \ref{LevineK}.

{\Small Acknowledgements: 
The first author thanks D. Ramakrishnan for asking the question of degeneration of the modified diagonal cycle, in August 2011 at IMSc. We also thank him  
for useful communications on the subject. We thank R. Sreekantan for pointing out an inaccuracy in a previous version.
This work was partly done during the first author's visits to Mainz in Dec 2011 and to FU Berlin and Mainz in Dec 2013, 
and is supported by SFB/TRR 45 of Deutsche Forschungsgemeinschaft.}

\section{Preliminaries}

We assume that the base field is $k=\comx$. In this section, we collect some facts and terminologies that we will need in the proofs.

Suppose $X$ is a smooth projective variety of dimension $d$ over $\comx$. We denote the rational higher Chow groups $CH^r(X,s)_\Q:= CH^r(X,s)\otimes \Q$. When $s=0$, 
$CH^r(X,0)=CH^r(X)$ is the usual Chow group of $X$.
We refer to Bloch's definition \cite{Bloch} of higher Chow groups and Levine's cubical version  \cite{Levine}.

\subsection{Indecomposable cycles}

There is a product map:
$$
CH^{r-m}(X,s-m) \otimes CH^1(X,1)^{\otimes m} \sta{\epsilon}{\rar} CH^r(X,s).
$$
The cokernel $\f{CH^r(X,s)}{{\rm Image}(\epsilon)}$  is the group of \textit{indecomposable} cycles.

Let $K_j$ denote Quillen's $K$-theory of coherent sheaves and $\cK_{j,X}$ denote the corresponding Zariski sheaf on $X$. Then we have the isomorphism:
$$
CH^r(X,s)\simeq H^{r-s}(X, \cK_{r,X}).
$$

An element of $CH^r(X,1)$ is written as a finite sum $\sum_i W_i \otimes h_i$, where $W_i\subset X$ are irreducible subvarieties of codimension $r-1$ and the rational 
functions $h_i\in \comx(W_i)^*$ satisfy $\sum_i {\rm div}(h_i)=0$ as a cycle on $X$.

\subsection{Regulator}

For a subring $A\subset \R$, denote $A(r):= (2\pi i)^rA \subset \comx$, and the Deligne cohomology $H^j_\cD(X,A(r))$ is the hypercohomology of the complex:
$$
A(r)_\cD: A(r) \rar \cO_X \rar \Omega^1_X \rar...\rar \Omega^{r-1}_X.
$$
There is a short exact sequence:
$$
0\rar \f{H^{2r-2}(X,\comx)}{F^rH^{2r-2}(X,\comx) + F^rH^{2r-1}(X,\comx)} \rar H^{2r-1}_\cD (X,A(r)) \rar F^r H^{2r-1}(X, A(r)) \rar 0.
$$
%
There is a cycle class map, resp. regulator 
$$
c_{r,s}: CH^r(X,s) \rar H^{2r-s}_\cD(X, \Z(r)).
$$


\subsection{Levine's $K$-groups with supports}\label{LevineK}

In this subsection, we recall Levine's $K$-groups \cite[\S 8, p.48-51]{Levine2} with supports, and their relation to higher Chow groups.

Let $X$ be a quasi-projective scheme over a field $k$ and $K(X)$ denote the Quillen $K$-theory space (or spectrum) of $X$.
The $K$-groups $K_j(X)$ are the $j$-th homotopy groups of $K(X)$.

1) There is a natural $\lambda$-ring structure on $K_*(X)$.

2) If $U\subset X$ is an open subset, then we have the $K$-theory with supports $K^W_X$ in $W:=X-U$, defined as the homotopy fiber of the natural map:
$$
K(X)\rar K(U).
$$

3) There is a long exact sequence:
$$
K^W_p(X)\rar K_p(X) \rar K_p(X-W) \rar K_{p-1}^W(X) \rar ...
$$
associated to the fiber sequence 
$$
K^W(X) \rar K(X) \rar K(U).
$$
4) If $Y_1,Y_2,...,Y_n \subset X$ are closed subschemes, the $n$-cube of spaces is denoted $$K(X,Y_1,Y_2,...,Y_n)_*$$ with
$$
K(X;Y_1,...Y_n)_I:= K(\cap_{i\in I}Y_i), \, I\subset \{1,...,n\}.
$$
The relative $K$-theory space $K(X; Y_1,...,Y_n)$ is defined as:
$$
K(X; Y_1,...,Y_n):= \m{holim}_{[0,1]^n} (K(X; Y_1,...,Y_n)_*).
$$
Here holim denotes the homotopy limit.

5) If $U\subset X$ is an open subset and $W=X-U$, the relative $K$-theory space with supports, $K^W(X; Y_1,...,Y_n)$ is defined as the homotopy fiber of 
$$
K(X;Y_1,...,Y_n) \rar K(U; Y_1\cap U,...,Y_n\cap U).
$$
6) There is a natural $\lambda$-operation on the relative $K$-groups with support $K^W_*(X,Y_1,...,Y_n)$ which satisfy the special $\lambda$-identities. 
The resulting Adams operations are denoted $\psi^k$.

7)  Recall the isomorphism, when $X$ is smooth and quasi-projective over a field:
$$
 CH^q(X,p)_\Q \simeq K_p(X)^{(q)}_\Q.
 $$
The RHS is the $q$-graded piece for the Adams operation $\psi^k$, on the $\gamma$-filtration.
In fact Levine obtains the above isomorphism, by inverting $(d+p-1)!$, where $d:=\dim(X)$.
 
8) Let $\Box^n$ be the affine space $\A^n_k$, let $\partial\Box^n$ be the collection of divisors 
$$
D^\epsilon_i\,:\, t_i=\epsilon,\, i=1,...,n; \epsilon =0,1,
$$
and let $\partial_0\Box^n= \partial\Box^n-D^0_n$.

9) Let $K_r^{[q]}(X\times \Box^p; X\times \partial\Box^p)$ be the direct limit
$$
K_r^{[q]}(X \times \Box^p; X\times \partial\Box^p):= \lim \limits_\rar K^W_r(X\times \Box^p; X\times \partial\Box^p)
$$
over closed subsets $W$ of $X\times \Box^p$ of codimension $q$, such that each component of $W$ intersects each face of $X\times \Box^p$ in codimension $q$. Similarly, 
$K_r^{[q]}(X\times \Box^p; X\times \partial_0\Box^p)$ be the direct limit as above with $\partial$ replaced by $\partial_0$.

10) There are canonical isomorphisms:
$$
K_0^{[q]}(X\times \Box^p; X\times \partial\Box^p)^{(q)}_\Q \,\simeq \, Z_p(\cZ^q(X,*)_{\Q})
$$
$$
K_0^{[q]}(X\times \Box^{p+1}; X\times \partial_0\Box^{p+1})^{(q)}_\Q \,\simeq \, Z_p(\cZ^q(X,p+1)_{\Q}).
$$
Here $\cZ^q$ denotes the Levine's cubical complex, and $Z_p$ denotes the (homological) $p$-dimensional cycles. 
 
11) The natural map 
$$
K_0^{[q]}(X\times \Box^p; X\times \partial\Box^p)^{(q)}_\Q \rar K_0(X\times \Box^p; X\times \partial\Box^p)^{(q)}_\Q 
$$ 
combined with 10), gives the map:
$$
Z_p(\cZ^q(X,*)_{\Q}) \rar K_0(X\times \Box^p; X\times \partial\Box^p)^{(q)}_\Q.
$$ 
 
12) The above map descends to give the  isomorphism:
$$
cl^{q,p}: CH^q(X, p)_{\Q} \sta{\simeq}{\rar} K_0(X\times \Box^p; X\times \partial \Box^p)^{(q)}_\Q \simeq K_p(X)^{(q)}_\Q.
$$
In fact with denominators $\f{1}{(d+p-1)!}$.

\section{Modified diagonal cycle on a triple product of a curve}

Suppose $X$ is a smooth projective connected curve of genus $g$ defined over the complex numbers.
Consider the triple product $Y= X\times X\times X$. Gross and Schoen \cite{GrossSchoen} defined the modified diagonal cycle as follows:
fix a closed point $e\in X$. Define the following subvarieties of codimension $2$ on $Y$.
\begin{eqnarray*}
 \Delta_{123} & := & \{x,x,x): x\in X\}\\
 \Delta_{12} & := & \{(x,x,e): x\in X\} \\
 \Delta_{13} & := & \{(x,e,x): x\in X\} \\
 \Delta_{23} & := & \{(e,x,x): x\in X\}\\
 \Delta_1 & := & \{(x,e,e): x\in X\} \\
 \Delta_2 & := & \{(e,x,e): x\in X \} \\
 \Delta_3 & := & \{(e,e,x): x\in X\}
 \end{eqnarray*}
 
The modified diagonal cycle is defined as follows:
$$
\Delta_e := \Delta_{123} -\Delta_{12}-\Delta_{13}-\Delta_{23} + \Delta_{1}+ \Delta_2 + \Delta_3.
$$

It was shown in \cite[Proposition 3.1,p.654]{GrossSchoen} that the cycle $\Delta_e$ is homologous to zero. Furthermore, if the genus $g(X)=0$, 
then $\Delta_e$ is rationally equivalent to zero and when the curve $X$ is hyperelliptic and the point $e$ is a Weierstrass point then $\Delta_e$ is 
a torsion element in $CH_1(Y)$ of order $6$ \cite[Proposition 4.8,p.658]{GrossSchoen}.

\subsection{A projector $P_e$ on the product variety}

Fix a closed point $e\in X$.  A projector $P_e$ is defined in \cite[p.652]{GrossSchoen} which acts on the Chow group of $Y=X \times X\times X$.
For any ordered subset $T\subset \{1,2,3\}$, let $T'$ be the complementary subset $\{1,2,3\}-T$, and $|T|$ denote the cardinality of $T$. Write
$p_T:X\times X\times X\rar X^{|T|}$ for the usual projection and let $q_T:X^{|T|}\rar X\times X\times X$ be the inclusion, with $e$ inserted at the missing coordinates.
Let $P_T$ be the graph of the morphism $q_T\circ p_T:Y\rar Y$, viewed as a cycle of codimension $3$ on $Y\times Y$.  Define
$$
P_e= \sum_T (-1)^{|T'|} P_T\,\in\, Z^3(Y\times Y).
$$
In other words,
$$
P_e= P_{123}-P_{12}-P_{13}-P_{23} + P_1 +P_2 +P_3
$$

Then we have 
\begin{lemma}\label{annihilate}
The projector $(P_e)_*$ annihilates the cohomology groups $H^6(Y),H^5(Y),H^4(Y)$ and maps $H^3(Y)$ onto the K\"unneth summand $H^1(X)\otimes H^1(X)\otimes H^1(X)$.
\end{lemma}
\begin{proof}
See \cite[Corollary 2.6, p.654]{GrossSchoen}.
\end{proof}

The modified diagonal cycle can now be written as:
$$
\Delta_e= (P_e)_*\Delta_X
$$
where $\Delta_X\subset X\times X\times X$ is the diagonal $\{(x,x,x): x\in X\}$.

Consider the natural morphism, fixing a basepoint $p\in X$.
$$
\begin{array}{ccccc}\label{summap}
f:Y & \rar & Sym^3X &\rar & J(X)\\
(a,b,c) &\mapsto & a+b+c & \rar & a+b+c -3p\\
\end{array}
$$

The Ceresa cycle in the Jacobian $J(X)$ is the cycle $X_e-X_e^-$, fixing the base point $e$ for the embedding $X\hookrightarrow J(X)$.

\begin{proposition}
There is an equality of the push-forward cycle with the Ceresa cycle, 
$$
f_*\Delta_e = 3(X_e-X_e^-)
$$
in the group of algebraic cycles modulo algebraic equivalence.
\end{proposition}
 
\begin{proof} See \cite[Proposition 2.9]{Colombo}.  
\end{proof}

\section{Specialization and Mayer-Vietoris exact sequence}

S. Bloch \cite{Bloch4} has defined specialization maps for higher Chow groups over a DVR $\cO$ with residue field $k$ and fraction field $F$. 
M. Levine \cite{Levine} has defined higher Chow groups $CH^*(X,*)$ for any variety $f: X \to B$ over a one-dimensional regular Noetherian scheme $B$.
We will assume that $B={\rm Spec}(\cO)$ for some local $k$--algebra and DVR $\cO$ with residue field $k$, 
and combine both results to obtain specialization maps in Levine's context.

Higher Chow groups are defined as homology groups of a certain complex $\cZ^*(X,*)$ defined in \cite{Levine}. Levine shows a localization sequence 
(distinguished triangle) 
$$
0 \to \cZ^{p-r}(Z,*) \to \cZ^p(X,*) \to \cZ^p(U,*),  
$$
where $Z \subset X$ is a closed subscheme of pure codimension $r$ and $U$ the open complement. Here, one has to be careful about the notion of 
dimension of a cycle, see \cite{Levine}. The induced coboundary map will be denoted by 
$$
\partial: CH^*(U,*) \to CH^{*-1}(Z,*-1). 
$$

Let $0 \in B$ be the closed point and $\pi$ a uniformizing parameter of $B$. We will apply this result in the case where
$Z=f^{-1}(0)$ is the closed fiber of $f$, and $r=1$.

There is a canonical class $[\pi] \in f^* CH^1(B \setminus \{0\},1) \subset CH^1(U,1)$ induced by the isomorphism
$$
CH^1(Y,1)=H^0(Y,{\mathcal O}_Y^*)
$$  
for the smooth variety $Y=B \setminus\{0\}$. 

We define the specialization morphism 
\begin{equation}\label{specializationmap}
\sp: CH^*(U,*) \to CH^*(Z,*) 
\end{equation}
via 
$$
\sp(W):=\partial(W \cdot [\pi]).
$$
The following is due to S. Bloch \cite{Bloch4} and W. Fulton \cite[pg. 398]{Fulton}:

\begin{proposition}\label{specialization}
The specialization map $\sp$ is a $k$--algebra homomorphism. 
\end{proposition}

Note that $\sp$ agrees with Fulton's specialization map $\sigma$ for $n=0$, and that for all cycles on $U$ which are restrictions of cycles
from $X$, specialization is simply given by taking the restriction from $X$ to the closed fiber over $Z$. 

\begin{proof}
We follow Bloch \cite{Bloch4}. Denote by $\cN_Z$, $\cN_X$ and $\cN_U$ the cycle DGA's $\cZ^*(Z,*)$, $\cZ^*(X,*)$ and $\cZ^*(U,*)$ respectively.
$\cN_X$ also has a multiplicative structure, but only in the derived category, since one uses a moving lemma.  There are natural restriction maps 
$$
\alpha: \cN_X \to \cN_U, \quad i^*: \cN_X \to \cN_Z. 
$$
There is also a pullback homomorphism $p^*: \cN_Z \to \cN_X$ and the coboundary $\partial: \cN_U \to \cN_Z[-1]$.
First one defines a section $\vartheta_\pi$ to $\partial$ by setting 
$$
\vartheta_\pi:= (\text{mult. with } [\pi]) \circ \alpha \circ p^*: \cN_Z[-1] \to \cN_U.
$$
This satisfies $\vartheta_\pi \circ \partial = id$ on $\cN_Z[-1]$ and induces an additive splitting 
$$
\alpha \oplus \vartheta_\pi: \cN_X \oplus \cN_Z[-1] \to \cN_U
$$
in the derived category. As in \cite{Bloch4} it follows that 
$$
\sp \circ \alpha= i^*: \cN_X \to \cN_Z.
$$
To finish the proof, one verifies the following formulas as in \cite{Bloch4}:
$$
\partial(x \cdot y)=\sp(x) \cdot \partial(y) + (-1)^{\deg(y)} \partial(x) \cdot \sp(y),
$$
$$
\partial(\alpha(x) \cdot \vartheta_\pi(y))=i^*(x) \cdot y= \sp(\alpha(x)) \cdot \partial(\vartheta_\pi(y)).
$$
\end{proof}

\subsection{Mayer-Vietoris spectral sequence} 

Let $Z$ be a union of $m$ components $Z_1,...,Z_m$ of the same dimension. Define
$$ 
Z^{[0]}={\coprod}_{1\leq i \leq m} Z_{i} \text{ and }
Z^{[1]} = {\coprod}_{1\leq i<j \leq m} Z_{i} \cap Z_{j}.
$$
Then there is an inclusion $i: Z^{[1]} \hookrightarrow Z^{[0]}$.

\begin{proposition}\label{MayerVietoris}
 We have the exact sequence
$$ CH^{r}(Z^{[0]},p)\to CH^{r}(Z,p) {\longrightarrow } 
{\rm Ker}\big(CH^{r-1}(Z^{[1]},p-1){\buildrel i_* \over \longrightarrow }  CH^{r}(Z^{[0]},p-1) \big) \to 0,$$
\end{proposition} 

\begin{proof} This is part of the lower term sequence for the 3rd quadrant Mayer-Vietoris spectral sequence \cite{SaitoSMS}
$E^{a,b}_1= CH^{a+r}(Z^{[-a]},-b) \Rightarrow CH^{r}(Z,-a-b)$. 
\end{proof}


\section{Degeneration of the modified diagonal cycle}\label{degcycle}

We recall that good models of families of triple products of curves have been constructed in \cite[p.664]{GrossSchoen}. 
We refer to \cite{GrossSchoen} for the definition of  'good family' or 'good models', but we mention that they are flat and 
proper families over a $Spec(R)$ where $R$ is a dvr, and the total space of the family is smooth. When $\cC \rar Spec(R)$ 
is a stable family of curves with generic fibre smooth, a good model of the triple product $\cC^3 \rar Spec(R)$ is constructed 
by successive blow-ups and the components of the special fibre is explicit described in \cite[p.664-665]{GrossSchoen}.
We do not reproduce the discussion, since in the later proofs our specialization maps are defined only on the open smooth 
subset of the total space of the family, and it suffices to know how the special fibre is described.

\subsection{Good family of triple products}\label{start}

Suppose $\cC \rar B:={\rm Spec}(R)$ is a stable family of genus $3$ curves over a DVR $R$, 
with a special  singular fibre $C_0$. Then a good family of triple products
$$
\pi:\cY\rar B 
$$ 
exists. If $C$ is irreducible of genus three, with one node, then the normalization  $C'$  is a smooth, hyperelliptic curve of genus $2$.  
We assume that the inverse image of the node in $C'$ are Weierstrass points. The special fibre $C_0$ of $\cC\rar B$ is $C'\cup \p^1$, 
such that $C'\cap \p^1$ consists precisely of the two Weierstrass points.

The special fibre $\cY_0:=\pi^{-1}(0)$ has $8$ components \cite[p.666,Example 6.15]{GrossSchoen}:
\begin{eqnarray*}
Y_1= C'\times C'\times C' \\
Y_2= C'\times C' \times \p^1\\
Y_3=C'\times \p^1\times C'\\
Y_4=C'\times \p^1 \times \p^1 \\
Y_5= \p^1\times C'\times C'\\
Y_6=\p^1\times C'\times \p^1 \\
Y_7=\p^1\times \p^1 \times C'\\
Y_8= \p^1\times \p^1\times \p^1.
\end{eqnarray*}

The normalization is given by $V_0=\cY_0^{[0]}=\sqcup_{i=1}^8 Y_i$.

\subsection{Extending the modified diagonal cycle}\label{diagsp}

Assume that the stable family of curves $\cX\rar B$ has  a section $\tilde{e}$ over $B$, such that on the special fibre it corresponds to a Weierstrass point on the component $C'$.

Let $Y=X\times X\times X$ be the generic fibre of $\cY\rar B$. The naive extension $\overline{\Delta_e}$ is obtained by taking the closures of each 
irreducible component $\Delta_{i}, \Delta_{i,j},\Delta_{i,j,k}$. 

Denote $\cU$ the complement $\cY-\cY_0$. Consider the specialization map \eqref{specializationmap}:
$$
sp:CH^2(\cU)\rar CH^2(\cY_0).
$$
The image of $\Delta_e$ under the map $sp$ is the same as the restriction of $\overline{\Delta_e}$ to $\cY_0$.

In our situation, the Mayer-Vietoris sequence in Proposition \ref{MayerVietoris} gives a right exact sequence:
$$
 \rar CH^2(V_0) \rar CH^2(\cY_0) \rar \m{Ker}(CH^1(\cY_0^{[1]},-1)\sta{i_*}{\rar} CH^2(V_0,-1))\rar 0.
$$ 
Since the right term is zero, the specialization cycle $\Delta^0_e$ lifts to a cycle on $V_0$. We denote this cycle again by $\Delta^0_e$.

\subsection{Higher Chow cycle}

Note that the cycle $\overline{\Delta_e}$ can be written as:
$$
\overline{\Delta_e} := \overline{\Delta_{123}} -\ov{\Delta_{12}}-\ov{\Delta_{13}}-\ov{\Delta_{23}} + \ov{\Delta_{1}}+ \ov{\Delta_2} + \ov{\Delta_3}.
$$
Here
\begin{eqnarray*}
 \ov{\Delta_{123}} & := & \{x,x,x): x\in \cX\}\\
 \ov{\Delta_{12}} & := & \{(x,x,\tilde{e}): x\in \cX\} \\
 \ov{\Delta_{13}} & := & \{(x,\tilde{e},x): x\in \cX\} \\
 \ov{\Delta_{23}} & := & \{(\tilde{e},x,x): x\in \cX\}\\
 \ov{\Delta_1} & := & \{(x,\tilde{e},\tilde{e}): x\in \cX\} \\
 \ov{\Delta_2} & := & \{(\tilde{e},x,\tilde{e}): x\in \cX \} \\
 \ov{\Delta_3} & := & \{(\tilde{e},\tilde{e},x): x\in \cX\}.
 \end{eqnarray*}

We have the smooth families:
$$
\cU\rar W,\, \cJ(\cC_W)\rar W
$$
associated to the smooth projective family $\cC_W\rar W:=B-\{0\}$ of genus three curves. The first one is the family of triple products and the second is the smooth Jacobian family.
There is a proper morphism 
$$f:\cU\rar \cJ(\cC_W)
$$ over $W$, with respect to the section $\tilde{e}$. On the special fibre, the morphism extends and factors via the triple product of the singular nodal curve $C$  
whose normalization is $C'$ (see \eqref{start}), to the Jacobian of the single nodal curve. This is a generalised Jacobian, i.e., an extension of $J(C')$ by $\C^*$. 
The compactification is a $\p^1$-bundle over $J(C')$. 

\begin{lemma}
The special fibre of the Jacobian family admits a degree two map to the trivial extension $J(C')\times \p^1$. 
\end{lemma}

\begin{proof} Since the node on $C$ corresponds to two Weierstrass points on $C'$, in the exact sequence
\[
{\rm Hom}(J(C')_2, \C^*) \to {\rm Ext}^1(J(C'),\C^*) \to {\rm Ext}^1(J(C'),\C^*)
\]
the extension element of the special fibre in the middle group maps to zero in the group on the right.
This corresponds to a degree two isogeny between the special fiber and $J(C')\times \p^1$.
\end{proof}

Hence, one has a diagram of pushforward maps:

$$
\begin{array}{ccc}
Z^2(\cU)_\Q  & {\buildrel {\rm sp} \over \rar} &  Z^2(\cY^{\bullet})_\Q\\
f_* \downarrow && \downarrow g_*\\
Z^2(\cJ(\cC_W))_\Q &  {\buildrel {\rm sp} \over \rar}  & Z^2(J(C') \times \p^1)_\Q \\
\end{array}
$$
Here we note that $g_*$ is in fact composed with the pushforward onto $J(C') \times \p^1$, via the degree two morphism mentioned above. 
Similarly, the second horizontal map $\rm{sp}$ is actually composed with pushforward onto $J(C')\times \p^1$.

Here $Z^2(\cY^\bullet)$ denotes the group of codimension two cycles on the simplicial set $\cY^\bullet$. 
 
\begin{lemma} 
The diagram commutes, up to rational equivalence. 
\end{lemma}
\begin{proof}
The pushforward map $g_*$ is induced by a finite morphism $V_0=\cY_0^{[0]} \rar J(C') \times \p^1$. 
On the component $Y_1$ it is given by $f': C' \times C' \rar J(C')$ and the hyperelliptic
map $C' \rar \p^1$. On the $Y_2$, $Y_3$ and $Y_5$ it is $f' \times {\rm id}$, and degenerate on 
the other components, since there is no non-constant morphism from $\p^1$ to $J(C')$.  
Both $f_*$ and $g_*$ make sense in families and the diagram commutes with specialization up to rational equivalence.
\end{proof}

\begin{definition}
$\Delta_{e,1}:= {\rm sp} (\Delta_e)$. 
\end{definition}

\begin{lemma} The specialization $\Delta_{e,1}$ is supported  on $Y_1$, $Y_2$, $Y_3$ and $Y_5$. 
The restriction to $Y_1$ vanishes in $CH^2(Y_1)_\Q$, if $e$ is a Weierstrass point on $C'$.
\end{lemma}

\begin{proof}
The specialization is a codimension two cycle on the simplicial set $\cY^\bullet$. However, it is supported on the skeleton $\cY^{[0]}$ of the simplicial set. Furthermore, 
by the Mayer-Vietoris sequence in \S \ref{diagsp}, it lies in the normalization $V_0$. But $V_0$ is a disjoint union of eight components 
$Y_1,...,Y_8$, see \S \ref{start}. 
Hence, we consider the specialization on  each component $Y_i$, as follows.

The specialization of $\Delta_{e,1}$ to $Y_1$ is the restriction of (sum of) the closure in $\cY$ of each diagonal components of the modified diagonal cycle on $\cU$, to $Y_1$. 
The restrictions of the closures correspond to the diagonal component of the modified diagonal cycle $\Delta_e$ on $C'\times C'\times C'$, 
and $e$ is a Weierstrass point on $C'$. In other words, the cycle ${\rm sp}(\Delta_{e,1})$ is the modified diagonal cycle on $C'\times C'\times C'$. Hence, 
by \cite[Proposition 4.8, p.658]{GrossSchoen}, 
we conclude that this class is rationally equivalent to zero. The restrictions to $Y_4$, $Y_6$, $Y_7$ and $Y_8$ are zero
as $g$ is constant on these components. More concretely, $g_*$ of the specialization on these components is the zero map, since there are no non-constant maps 
from $\p^1$ to $J(C')$. Furthermore, $g^*g_*$ on $\Delta_{e,1}$ is a multiple of $\Delta_{e,1}$. 
Hence triviality of $g_*$ on $Y_4,Y_6,Y_7,Y_8$ implies that the specialization on these components is torsion.
\end{proof}

\begin{lemma}\label{subChow}
The pushforward cycle $g_* \Delta_{e,1}$ has a representative in the subgroup
$$
Z^2_\partial(J(C') \times \p^1)_\Q:=\ker\left((i_0^*,i_\infty^*): Z^2(J(C') \times \p^1)_\Q \rar \oplus^2 Z^2(J(C'))_\Q\right).
$$
Therefore, $g_* \Delta_{e,1}$ defines a higher Chow cycle in $CH^2(J(C'),1)_\Q$.
\end{lemma}
\begin{proof} This is because the pushforward $g_*$ (onto $J(C')\times \p^1$) of $\Delta_{e,1}$ is supported on the generalised Jacobian, which is a trivial $\C^*$-extension of $J(C')$.
\end{proof}

We will now use the following results, which are essentially due to Colombo and van Geemen \cite{Colombo, Colombo-van-Geemen} and Green-Griffiths-Kerr \cite{GGKerr}: 

\begin{proposition}{~} \\
(a) The cycle $g_* \Delta_{e,1}$ is a degeneration of Ceresa's cycle up to algebraic equivalence. \\
(b) The Abel-Jacobi class of Ceresa's cycle specializes to the higher Abel-Jacobi invariant of Collino's cycle $(C'_p,h_p)+(C'_q,h_q)$ in 
$CH^2(C'\times C',1)_\Q$, as defined in \cite{Collino}.
\end{proposition}

\begin{proof} (a) By \cite{Colombo-van-Geemen}, $f_* \Delta_e$ is $3$ times the Ceresa cycle up to algebraic equivalence. 
Hence, ${\rm sp}(f_* \Delta_e)=g_* \Delta_{e,1}$ is a degeneration of the Ceresa cycle up to algebraic equivalence.

(b) The Abel-Jacobi invariant of Ceresa's cycle $C-C^-$, is most conveniently described as an integration current $\int_\Gamma -$, where
$\Gamma$ is a $3$-chain on $J(C)$, with $\partial \Gamma=C-C^-$. This is a current in the relative intermediate Jacobian $\cD^3(\cJ(\cC))$ (fiberwise) of the family 
$\cJ(\cC_W)\rar W$, resp. a current in
$\cD^3(\cU)$ which can be evaluated on $3$-forms with compact support. 
Colombo \cite[pg. 788-789]{Colombo} has observed that the specialization map on the level of currents is given by
${\rm sp}(\int_\Gamma -)=\int_{\tilde \Gamma} -$, where $\tilde \Gamma$ is a $3$-chain on $C' \times C' \times \C^*$. 
Choosing test forms $\alpha \wedge \frac{dz}{z}$ with compact support, this defines a $2$-current in $\cD^2(J(C'))$ 
which is the regulator value for Collino's cycle as Colombo indicates in loc. cit.. \end{proof} 

We point out that an alternative proof for (b) can be given along the lines of \cite[IV.D]{GGKerr}, 
in which the limit of Abel-Jacobi maps in families is studied. \\

We can now prove our main result (see Theorem \ref{maintheorem} in the introduction): 

\begin{theorem}\label{degener}
The degeneration $\Delta_{e,1}$ is a higher Chow cycle in $CH^2(C'\times C',1)_\Q$. It is indecomposable, 
if $C'$ is generic, and its regulator image is non-zero.
\end{theorem}

\begin{proof}
Replace Chow groups by (relative) $K_0$. We refer to Levine's paper \cite[\S 8, p.48-51]{Levine2}, see \S \ref{LevineK}. 
We use the isomorphism of higher Chow group with the 
relative Chow group with supports, see \S \ref{LevineK} 12). Then we note that in fact, 
$g_* \Delta_{e,1}$ is in the image of 
$$
K_0(J(C') \times \Box^1,J(C') \times \partial \Box^1)^{(2)}_\Q \rar K_0(J(C') \times \Box^1)^{(2)}_\Q.
$$
Denote the cycle in $CH^2(J(C'),1)_\Q\simeq K_0(J(C') \times \Box^1,J(C') \times \partial \Box^1)^{(2)}_\Q$ by $W$, which maps to $g_*\Delta_{e,1}$.

Then we use the commutative diagram
$$
\begin{array}{ccc}
K_0(C' \times C' \times \Box^1,C'\times C' \times \partial \Box^1)^{(2)}_\Q & \rar & K_0(C' \times C'\times \Box^1)^{(2)}_\Q.  \\
f_* \downarrow && \downarrow g_*\\
K_0(J(C') \times \Box^1,J(C') \times \partial \Box^1)^{(2)}_\Q &  \rar  & K_0(J(C') \times \Box^1)^{(2)}_\Q. \\
\end{array}
$$
Since $f$ is a finite proper morphism, the inverse image $(f_*)^{-1}(W)$ is a multiple of
$\Delta_{e,1}$, and it lies in the relative $K_0$- group of $C'\times C'$. 
Now we use the isomorphism \S \ref{LevineK} 12): 
$$
CH^2(C' \times C',1)_\Q \cong K_0(C' \times C' \times \Box^1,C'\times C' \times \partial \Box^1)^{(2)}_\Q,
$$
to deduce that $\Delta_{e,1}$ corresponds to a higher Chow cycle in $CH^2(C'\times C',1)_\Q$.

The indecomposability and non-triviality of the regulator image of $\Delta_{e,1}$ follows, because the pushford $g_*\Delta_{e,1}$ is a multiple of 
Collino's regulator indecomposable higher Chow cycle on $J(C')$, up to algebraic equivalence. \end{proof}

\begin{remark} 
Equivalently, the proof could have been done using \cite{GGKerr} and showing that the limit Abel-Jacobi map on the triple product family pushes forward to the limit 
Abel-Jacobi map of the Jacobian family. Namely, assuming after some finite base change that the family of curves has unipotent monodromy, 
and using \cite[Proposition 7.2]{GrossSchoen}, we know that the specialization is  
a numerically trivial cycle. Then by \cite{Lieberman}, we know that the specialization cycle is homologically trivial and using \cite{GGKerr}, 
the limit Abel-Jacobi map is well-defined. The pushforward is non-zero by \cite[IV D]{GGKerr}. Hence, the regulator image of $\Delta_{e,1}$ is non-zero.

One can also look at Griffiths' infinitesimal invariant of Ceresa's cycle. 
It can be computed from the cohomology class in $H^4(\cU)$ of $C-C^-$. The specialization of this
class is not homologous to zero on the closed fiber $V_0$. In Gross-Schoen \cite{GrossSchoen} this non-zero class is explained. 
It is the image of the infinitesimal invariant under specialization, and hence the infinitesimal invariant of $\Delta_{e,1}$
is non-zero. 
\end{remark}

\subsection{Second degeneration  higher Chow cycle $\Delta_{e,2}$.}

One can iterate this procedure, to obtain the following.

\begin{proposition}\label{elliptic}
The second degeneration of the cycle $\Delta_{e,1}\in CH^2(C'\times C',1)_\Q$ corresponds to an indecomposable higher Chow cycle $\Delta_{e.2}\in CH^2(E,2)_\Q$. 
Here $E$ is a generic elliptic curve and is the normalization of the degeneration, a single nodal curve, of a family of smooth hyperelliptic genus 
two curves over ${\rm Spec}(R)$, where $R$ is a DVR.
\end{proposition}

\begin{proof}
The proof is similar to the first degeneration $\Delta_{e,1}$. 
We consider a smooth family of genus two curves $\cC'\rar B-\{0\}$, $B:={\rm Spec}(R)$, where $R$ is a DVR. The special fibre is a single nodal curve whose normalization 
is an elliptic curve $E$, with the inverse image of the node being two Weierstrass points.  A good model $\cU$ of the double product smooth  family is considered as earlier, 
together with the higher Chow cycle in $CH^2(\cC'\times \cC',1)_\Q$. This is same as a cycle in 
$ Z^2_\partial(\cJ(\cC') \times \p^1)_\Q$, see Lemma \ref{subChow} (upto a finite map, as in the previous situation).   Using the specialization map $sp$, we obtain the degeneration $\Delta_{e,2}\in CH^2(E\times \p^1,1)_\Q$. 
In fact, using the relative $K_0$-groups of Levine (see proof Theorem \ref{degener}), it corresponds to a higher Chow cycle in $CH^2(E,2)_\Q$. This is a multiple of  
Collino's degenerated cycle, up to algebraic equivalence. Hence, it is indecomposable with non-zero regulator image.   
\end{proof}  

\subsection{Third degeneration higher Chow cycle $\Delta_{e,3}$.}

We can even further degenerate $\Delta_{e,2}$ to a cycle $\Delta_{e,3}$ in $CH^2(\C,3)\otimes \Q$ by the same method. 
This is related to Collino's work \cite[Remark 7.13]{Collino}. There also a formula for the regulator as indicated.


\end{document}